\pdfoutput=1
\documentclass{amsart}
\usepackage{colortbl}
\usepackage{amsmath,amsthm,amssymb}
\usepackage{graphicx}
\usepackage{harvard}

\numberwithin{equation}{section}

\def\ZZZ{{\mathbb Z}}
\def\LLL{{\mathcal L}}

\def\RRR{{\mathbb R}}

\def\RRR{{\Bbb R}}

\def\1{\mathbf 1}

\newtheorem{definition}{Definition}
\newtheorem{theorem}{Theorem}
\newtheorem{lemma}{Lemma}

\begin{document}

\title{Kernel Convergence Estimates for Diffusions with Continuous Coefficients}

\author{Claudio Albanese}

\email{claudio@level3finance.com}

\date{First version October 24th, 2007, last revision \today}

\thanks{ I am grateful to Paul Jones and Adel Osseiran for careful
reading of earlier versions of this paper. All remaining errors are
the author's own fault. }

\maketitle

\begin{abstract}

We are interested in the kernel of one-dimensional diffusion
equations with continuous coefficients as evaluated by means of
explicit discretization schemes of uniform step $h>0$ in the limit
as $h\to0$. We consider both semidiscrete triangulations with
continuous time and explicit Euler schemes with time step small
enough for the method to be stable. We find sharp uniform bounds for
the convergence rate as a function of the degree of smoothness which
we conjecture. The bounds also apply to the time derivative of the
kernel and its first two space derivatives. Our proof is
constructive and is based on a new technique of path conditioning
for Markov chains and a renormalization group argument. Convergence
rates depend on the degree of smoothness and H\"older
differentiability of the coefficients. We find that the fastest
convergence rate is of order $O(h^2)$ and is achieved if the
coefficients have a bounded second derivative. Otherwise, explicit
schemes still converge for any degree of H\"older differentiability
except that the convergence rate is slower. H\"older continuity
itself is not strictly necessary and can be relaxed by an hypothesis
of uniform continuity.

\end{abstract}

 \tableofcontents

\section{Introduction and Notations}

Consider a pair of backward and forward one-dimensional diffusion
equations of the form
\begin{equation}
{\partial \over \partial t} f(x; t) + {\mathcal L}^{0}_x f(x; t) = 0
, \;\;\;\;\;\;\;\;\;\;
 {\partial \over \partial t} g(y; t) = {\mathcal L}^{0*}_y
g(y; t) \label{eqdiff}
\end{equation}
where
\begin{equation}
{\mathcal L}^{0}_x = {1 \over 2} {\sigma(x)^2 } {\partial^2\over
\partial x^2} + \mu(x){\partial \over \partial x}.
\end{equation}
and its adjoint formally acts as follows:
\begin{equation}
({\mathcal L}^{0*}_y \phi)(y) =  {1\over 2} {\partial^2\over
\partial y^2} \big(\sigma(y)^2 \phi(y) \big) + {\partial \over
\partial y} \big(\mu(y) \phi(y)\big).
\end{equation}
on a test function $\phi$. These equations are defined on the
bounded interval $A = [-L,L] \subset \RRR$ where $0<L<\infty$. For
simplicity, we assume periodic boundary conditions and identify the
two boundary points $\pm L$ with each other.

In most of the paper, the coefficients $\sigma(x)$ and $\mu(x)$ are
assumed to be H\"older differentiable. More precisely, if $\alpha\in
(0,1]$, $k\in{\Bbb N}$ where ${\Bbb N} = \{0, 1, ...\}$ and the
function $\phi(x) \in \mathcal{C}^{k}(A)$ has $k\ge1$ continuous
derivatives, then one says that $\phi$ is H\"older differentiable of
order $(k, \alpha)$ if there exists a constant $c>0$ such that
\begin{equation}
\big\lvert \phi^{(k)}(x) - \phi^{(k)}(y) \big\lvert \leq c d(x,
y)^\alpha \;\;\;\;\label{eq_holder}
\end{equation}
uniformly for all $x, y \in A$. In case $k=0$ the function is called
H\"older continuous. The distance is defined consistently with the
assumed periodic boundary conditions and is given by
\begin{equation}
d(x, y) = \min_n \lvert x - y - 2 L n \lvert.
\end{equation}
The linear space of H\"older continuous or H\"older differentiable
periodic functions of order $(k, \alpha)$ on $A$ is denoted with
$\mathcal{C}^{k, \alpha}(A)$. We are interested in the case where
$\sigma^2\in\mathcal{C}^{k, \alpha}(A)$ and $\mu\in\mathcal{C}^{j,
\beta}(A)$ with $k+\alpha>0$ and $j+\beta>0$.

The hypothesis of H\"older continuity can be relaxed slightly by
assuming uniform continuity instead, i.e. that both $\sigma(x)^2$
and $\mu(x)$ satisfy a bound of the form
\begin{equation}
\big\lvert \phi(x) - \phi(y) \big\lvert \leq c \rho\big(d(x, y)\big)
\;\;\;\;\label{eq_holderrho}
\end{equation}
where $\rho(d)$ is a non-decreasing function such that
$\lim_{d\downarrow0}\rho(d) = 0$.

Let $u_0(x, y; t)$ denote a kernel of equation (\ref{eqdiff}), i.e.
a weak solution of the forward equation
\begin{equation}
{\partial \over \partial t} u_0(x, y; t) = {\mathcal L}^{0*}_y
u_0(x, y; t) \label{eqdiff}
\end{equation}
where the operator ${\mathcal L}^{0*}_y$ acts on the $y$ coordinate
and the following initial time condition is satisfied:
\begin{equation}
\lim_{t\downarrow 0} u_0(x, y; t) = \delta(x-y). \label{eqdiff}
\end{equation}
The kernel $u_0(x, y; t)$ formally satisfies also the backward
equation
\begin{equation}
{\partial \over \partial t} u_0(x, y; t) + {\mathcal L}^{0}_x u_0(x,
y; t) = 0  \label{eqdiff}
\end{equation}
where the operator ${\mathcal L}^0_x$ acts on the $x$ coordinate.

We are interested in existence, uniqueness, smoothness and
approximation schemes for the kernel $u_0(x, y; t)$, its first two
space derivatives with respect to the $x$ variable and its first
time derivative $\partial_t u_0(x, y; t)$. As a byproduct of this
analysis, we also find conclusions about the convergence of
${\mathcal L}^{0}_x u_0(x, y; t)$ and ${\mathcal L}^{0*}_y u_0(x, y;
t)$, as both expressions equal the first time derivative.

Diffusion equations are one of the single most studied topics in the
literature. Existence and uniqueness questions for the kernel were
address in \cite{Kolmogorov1931}, \cite{Feller1936}, \cite{Hille},
\cite{Yosida} and \cite{Ito57}. A classification of all the possible
boundary conditions is in \cite{Feller52}. The case of H\"older
continuous coefficients was resolved in \cite{Philips} based on
methods in \cite{Friedrichs} and \cite{LaxPhilips}. The hypothesis
of H\"older continuity was relaxed to uniform continuity in
\cite{FabesRiviere} and \cite{SV1969}. Strook and Varadhan also
introduce a new probabilistic framework where existence is proved by
reduction to the so-called martingale problem and a compactness
argument, thus shifting the attention from the kernel itself to the
underlying measure space.

The existence of a weak limit of continuous time Markov chains as
$h_m\downarrow 0$ was established in \cite{Sova} and \cite{Kurtz} by
using operator semigroup methods, see also the book
\cite{EthierKurtz} for a review. Convergence in the semigroup sense
takes place if the limit
\begin{equation}
(T_t \phi)(x)  = \lim_{m\to \infty, m\ge n} \sum_{y \in h_m Z \cap
A} u_{h_m}(x, y; t) \phi(y)
\end{equation}
exists for all test functions $\phi\in {\mathcal C}^{\infty}(A)$,
uniformly for all $x\in A_n, n\ge0$. A key result is that a
necessary and sufficient condition for this limit to exist and
define a semigroup $T_t$ is that generators converge also in the
same Banach space, i.e. that also the limit
\begin{equation}
\lim_{h_m\downarrow 0} {\mathcal L}^{h_m}_x \phi = {\mathcal
L}^{0}_x \phi
\end{equation}
exists in the uniform norm for all test functions $\phi\in {\mathcal
C}^{\infty}(A)$. See \cite{EthierKurtz} for a precise statement with
all the technical conditions and a proof. In \cite{SV} convergence
is reconsidered again by reduction to the martingale problem.

The problem has also been studied extensively in the numerical
analysis literature. Explicit and implicit Euler schemes where
coefficients are smooth and the data is rough in the sense that it
belongs to a $L^2$ space have been considered by several authors. In
the case that the Markov generator is symmetric and time
independent, one can make use of a spectral representation as in
\cite{BakerBrambleThomte} and with greater effort such methods may
also be used for more general situations, see \cite{Suzuki}. In
\cite{LuskinRannacher}, a parabolic duality argument is used to show
convergence for the standard Galerkin method. \cite{MingyouThomee}
use a simpler argument based on energy estimates. In \cite{Palencia}
one finds convergence bounds in maximum norm assuming the initial
condition is uniformly bounded and coefficients are constant.

In this article, we revisit this classic theme by considering the
problem of obtaining the kernel constructively as a limit of
increasingly fine triangulations schemes and in assessing the rate
of convergence with pointwise bounds on the kernel itself. More
precisely, let $h_m = L 2^{-m}, m \in {\Bbb N}$ and let $A_m
{h_m}\ZZZ\cap A$. Consider the sequence of operators
\begin{equation}
{\mathcal L}^{m}_x = {\sigma(x)^2 \over 2} \Delta^{m}_x +
\mu(x)\nabla^{m}_x. \label{eq_lm}
\end{equation}
defined on the $2^{m+1}$-dimensional space of all periodic functions
$f_{m}: A_m \to \RRR$, where
\begin{equation}
\nabla^{m}_x f(x) = {f(x+{h_m})-f(x-{h_m})\over 2 {h_m}}.
\end{equation}
and
\begin{equation}
\Delta^{m}_x f(x) = {f(x+{h_m})+f(x-{h_m})-2 f(x) \over {h_m}^2}
\end{equation}
These definitions also apply to the boundary points by periodicity.
We assume that $m\ge m_0$ where $m_0$ is the least integer such that
\begin{equation}
{\sigma^2(x)\over 2 h_m^2} > {\lvert \mu(x) \lvert \over 2h_m}
\label{def_diffusion_bounds}
\end{equation}
for all $m\ge m_0$ and all $x\in A_m$.

Let $u_{m}(x, y; t)$ denote the kernel of equation (\ref{eqdiff}),
i.e. the solution of the (forward) equation
\begin{equation}
{\partial \over \partial t} u_{m}(x, y; t) = {\mathcal L}^{m *}_y
u_{m}(x, y; t) \label{eqdiff}
\end{equation}
where the operator ${\mathcal L}^{m *}_y$ acts on the $y$ coordinate
and the following initial time condition is satisfied:
\begin{equation}
\lim_{t\downarrow 0} u_{m}(x, y; t) = \delta_{m}(x-y).
\label{eqdiff}
\end{equation}
Here,
\begin{equation}
\delta_{m}(x-y) =
\begin{cases}
{1 \over h_m} \;\;\;\;\;\;\;{\rm if} \;\; x = y \;\;\;{\rm mod} \;2 L\\
0 \;\;\;\;\;\;\;\;\;{\rm otherwise}.
\end{cases}
\label{eqdelta}
\end{equation}
Since (\ref{eqdiff}) is a finite system of linear ordinary
differential equations, the solution exists and is unique for all
times. The kernel $u_{m}(x, y; t)$ satisfies also the backward
equation
\begin{equation}
{\partial \over \partial t} u_{m}(x, y; t) + {\mathcal L}^{m}_x
u_{m}(x, y; t) = 0  \label{eqdiffb}
\end{equation}
where the operator ${\mathcal L}^{m}_x$ acts on the $x$ coordinate.
Using functional calculus notations for the exponential of a matrix,
we also have that
\begin{equation}
u_{m}(x, y; t) = {1 \over h_m} \exp\big( t {\mathcal L}^{m} \big)
(x, y).
\end{equation}

Our main result can be stated as follows:
\begin{theorem} \label{thconv1}
Suppose that  $\sigma^2 \in {\mathcal C}^{k, \alpha}$  and $\mu \in
{\mathcal C}^{j, \beta}$ and let
\begin{equation}
\gamma = \min\{ 2, k+\alpha, j+\beta  \}.
\label{eq_gamma}
\end{equation}
Assume that $\gamma > 0$. Then there is a constant $c>0$ such that
for all $m'\ge m \ge m_0$ and all $x, y \in A_{m}$ the following
inequalities hold:
\begin{itemize}
\item[(i)]
\begin{equation}
\lvert u_m(x,y; t) - u_{m'}(x, y; t) \lvert \leq c h_m^\gamma
\label{bm_uest}
\end{equation}
\item[(ii)]
\begin{align}
\lvert \partial_t & u_m(x, y; t) - \partial_t u_{m'}(x, y; t)
\lvert \notag \\
& = \lvert {\mathcal L}^{m*}_y u_m(x, y; t) - {\mathcal L}^{m'*}_y
u_{m'}(x, y; t) \lvert  \notag
 \\
& = \lvert {\mathcal L}^{m}_x u_m(x, y; t) - {\mathcal L}^{m'}_x
u_{m'}(x, y; t) \lvert \leq c h_m^\gamma. \notag
 \\
 \label{bm_gest}
\end{align}
\end{itemize}
\end{theorem}

A version of this theorem under slightly weaker conditions can be
formulated as follows:
\begin{theorem} \label{thconv2}
Suppose that  $\sigma(x)^2$  and $\mu(x)$ are uniformly continuous
functions in $A$. Let the function $\rho(d)$ be non-decreasing and
be such that $\lim_{d\downarrow0} \rho(d) = 0$ and equation (\ref
{eq_holderrho}) holds. Then there is a constant $c>0$ such that for
all $m'\ge m \ge m_0$ and all $x, y \in A_{m}$ the following
inequalities hold:
\begin{itemize}
\item[(i)]
\begin{equation}
\lvert u_m(x, y; t) - u_{m'}(x, y; t) \lvert \leq c \rho(h_m)
\label{bm_uest2}
\end{equation}
\item[(ii)]
\begin{align}
\lvert \partial_t & u_m(x, y; t) - \partial_t u_{m'}(x, y; t)
\lvert \notag \\
& = \lvert {\mathcal L}^{m*}_y u_m(x, y; t) - {\mathcal L}^{m'*}_y
u_{m'}(x, y; t) \lvert  \notag
 \\
& = \lvert {\mathcal L}^{m}_x u_m(x, y; t) - {\mathcal L}^{m'}_x
u_{m'}(x, y; t) \lvert \leq c \rho(h_m). \notag
 \\
\label{bm_gest2}
\end{align}
\end{itemize}
\end{theorem}

Next, we consider the case where also time is discretized and prove
the following result:

\begin{theorem}\label{theo_diffdtker}
Suppose that  $\sigma^2$  and $\mu$ satisfy equations of the form
(\ref {eq_holderrho}) with a non-decreasing function $\rho(d)$ such
that $\lim_{d\downarrow0} \rho(d) = 0$. Consider the discretized
kernel
\begin{equation}
u^{\delta t}_m(x, y; t) = h_m^{-1} \left( 1 + \delta t {\mathcal
L}^m \right)^{t\over\delta t} (x, y; t).
\end{equation}
where ${\mathcal L}^m$ is the operator in (\ref{eq_lm}) and $\delta
t_m $ is so small that
\begin{equation}
\min_{x\in A_m} 1 + \delta t_m {\mathcal L}^m(x, x) > 0
\end{equation}
Assume that boundary conditions are periodic and that the ratio
${t\over \delta t} = N$ is an integer. Then here is a constant $c>0$
such that the following bounds hold for all $m\ge m_0$ and all $x, y
\in A_{m}$:
\begin{itemize}
\item[(i)]
\begin{equation}
\lvert  u_m(x, y; t) - u_m^{\delta t}(x, y; t)
 \lvert \leq c h_m^2
\end{equation}
\item[(ii)]
\begin{align}
\bigg\lvert \partial_t & u_m(x, y; t) - {u_m^{\delta t}(x, y;
t+\delta t) - u_m^{\delta t}(x, y; t) \over \delta t}
\bigg\lvert \notag \\
& = \lvert {\mathcal L}^{m*}_y u_m(x, y; t) - {\mathcal L}^{m*}_y
u_{m}^{\delta t}(x, y; t) \lvert  \notag
 \\
& = \lvert {\mathcal L}^{m*}_x u_m(x, y; t) - {\mathcal L}^{m*}_x
u_{m}^{\delta t}(x, y; t) \lvert  \leq c h^2 \notag
 \\
 \label{bm_gest}
\end{align}
\end{itemize}
\end{theorem}

The paper is organized as follows. In Section 2 we consider the case
of Brownian motion and review a result in \cite{AlbaneseMijatovicBM}
which establishes the theorems above in this simple particular case
where Fourier analysis in the space direction can be used to carry
out a precise calculation. In Section 3, we consider the case of a
diffusion where both the volatility and the drift have two bounded
derivatives. In this case, we make use of time-homogeneity and carry
out a Fourier transform in the time direction after path
conditioning. In Section 4, we extend the derivation to the case of
non-smooth coefficients. Section 5 is dedicated to the case where
time is discretized and we prove Theorem \ref{theo_diffdtker}.

\section{Constant Coefficients}

In this Section, we prove Theorem 1 in the special case where the
volatility and the drift coefficients are constant, i.e.
\begin{equation} {\LLL^m}_x = \mu
\nabla^m_x + {1\over 2} \sigma^2 \Delta^m_x. \label{def_brownian}
\end{equation}

It suffices to consider the case $m' = m+1$. Let $B_m$ be the
Brillouin zone defined as follows:
\begin{equation}
B_m = \left\{ -{2^{m-1} \pi \over L} + {k \pi \over L } , k = 0, ..
2^m-1 \right\} \label{brillouin}
\end{equation}
Let ${\mathcal F}_m: \ell^2(A_m) \to \ell^2(B_m)$ be the Fourier
transform operator defined so that:
\begin{equation}
\hat f(p) \equiv {\mathcal F}_m(f)(p) = h_m \sum_{x\in A_m} f(x) e^{
- i p x}
\end{equation}
for all $p\in B_m$. The inverse Fourier transform is given by
\begin{equation}
{\mathcal F}_m^{-1} (\hat f)(x) =  {1 \over 2 L} \sum_{p\in B_m}
\hat f(p) e^{i p x}.
\end{equation}

The Fourier transformed generator is diagonal and is given by the
operator of multiplication by
\begin{equation}
\hat {\ell}^m (p) = {\mathcal F}_m {\LLL^m} {\mathcal F}_m^{-1}(p,
p) = -i \mu {\sin h_m p \over h_m} + {\sigma^2 } {\cos h_m p - 1
\over h_m^2}.
\end{equation}
We have
\begin{equation}
u_m(x, y; t) = {1 \over 2 L} \sum_{p\in B_m} e^{t \hat {\ell^m} (p)}
e^{i p (y - x) }.
\end{equation}
Using this Fourier series representation, we find
\begin{align}
& \big\lvert u_m(x, y; t) - u_{m+1}(x, y; t) \big\lvert \notag \\
&\leq {1 \over 2 L} \bigg\lvert \sum_{p\in B_m} \bigg( e^{t {\hat
\ell^m} (p)} - e^{t {\hat \ell^{m+1}} (p)} \bigg) e^{i p (y - x)
}\bigg\lvert + {1 \over 2 L} \bigg\lvert \sum_{p\in B_{m+1}\setminus
B_m } e^{t {\hat \ell^{m+1}} (p)} e^{i p (y - x) }\bigg\lvert.
\notag \\
\end{align}

Let
\begin{equation}
K_m = \sqrt{ \lvert \log h_{m+1} \lvert\over \sigma^2 t}.
\label{bm_km}
\end{equation}
If $h_m$ is small enough, i.e. if $m_0$ is sufficiently large, we
have that
\begin{align}
{1 \over 2 L} \bigg\lvert \sum_{p\in B_{m}, \lvert p \lvert \ge K_m
} e^{t {\hat \ell^{m}} (p)} e^{i p (y - x) }\bigg\lvert \leq {1
\over 2 L} \sum_{p\in B_{m}, \lvert p \lvert \ge K_m } e^{t \Re(
\hat\ell^{m} (p))} \leq c \exp\bigg( t \sigma^2 {\cos h_m K_m -
1\over h_m^2} \bigg) \leq c h_m^2. \notag \\
\end{align}
where $\Re(a)$ denotes the real part of $a\in {\mathbb C}$ and  $c$
denotes a generic constant. Similarly
\begin{align}
{1 \over 2 L} \bigg\lvert \sum_{p\in B_{m+1}, \lvert p \lvert \ge
K_{m} } e^{t {\hat \ell^{m+1}} (p)} e^{i p (y - x) }\bigg\lvert \leq
{1 \over 2 L} \sum_{p\in B_{m}, \lvert p \lvert \ge K_m } e^{t \Re(
\hat\ell^{m+1} (p))} \notag \\
\leq c \exp\bigg( t \sigma^2 {\cos h_{m+1} K - 1\over h_{m+1}^2}
\bigg) \leq c h_{m+1}^2 \notag \\
\end{align}
Since
\begin{align}
{1 \over 2}h^2 p^3 - {1\over8} h^4 p^5 \leq {\sin hp \over h} -
{\sin 2hp \over 2h} \leq {1 \over 2} h^2 p^3
\end{align}
and
\begin{align}
-{1 \over 8}h^2 p^4 \leq {\cos hp -1 \over h^2} - {\cos 2hp -1 \over
(2h)^2} \leq - {1\over8} {h^2 p^4} + {1\over 48} h^4 p^6 .
\end{align}
we find that if $\lvert p \lvert \leq {\sqrt{2} \over h }$ then
\begin{align}
\lvert \hat {\ell}^m(p) - \hat {\ell}^{m+1}(p) \lvert \leq {\mu\over
4} h^2 \lvert p \lvert^3 + {\sigma^2 \over16} h^2 p^4.
\end{align}
Moreover, since
\begin{align}
- {1 \over 2} p^2 \leq {\cos hp -1 \over h} \leq  - {1 \over 2} p^2
+ {1\over 24} h^2 p^4
\end{align}
we conclude that in case $\lvert p \lvert \leq h^{-1} \sqrt {2 \over
3}$, the following inequality holds:
\begin{align}
{\cos hp -1 \over h} \leq - {1\over4} {p^2}
\end{align}
Hence, if  $m_0$ is large enough, we find
\begin{align}
{1 \over 2 L} \bigg\lvert \sum_{p\in B_m, \lvert p \lvert \leq K}
\bigg( e^{t {\hat \ell^m} (p)} - & e^{t {\hat \ell^{m+1}} (p)}
\bigg) e^{i p (y - x) }\bigg\lvert \notag \\
& \leq {1 \over 2 L}  \sum_{p\in B_m, \lvert p \lvert \leq K} e^{-
{1\over4} {p^2}}
\bigg( e^{{\mu t\over 4} h_m^2 \lvert p \lvert^3 + {\sigma^2 t\over16} h_m^2 p^4 } - 1\bigg) \notag \\
& \leq {1 \over 2 L}  \sum_{p\in B_m, \lvert p \lvert \leq K} e^{-
{1\over4} {p^2}} \bigg( {\mu t\over 4} h_m^2 \lvert p \lvert^3 +
{\sigma^2 t\over16} h_m^2 p^4 \bigg) \leq c h_m^2
\end{align}
for some constant $c>0$ independent of $m$. This concludes the proof
of convergence for the kernel in the special case of constant
coefficients.

To estimate the first derivative, notice that
\begin{equation}
\nabla u_m(x, y; t) = {1 \over L} \sum_{p\in B_m} e^{t\hat {\ell^m}
(p)} {\sin p h_m \over h_m} e^{i p (y - x) }.
\end{equation}
and
\begin{align}
& \big\lvert u_m(x, y; t) - u_{m+1}(x, y; t) \big\lvert \notag \\
&\leq {1 \over 2 L} \bigg\lvert \sum_{p\in B_m} \bigg( e^{t {\hat
\ell^m} (p)}{\sin p h_m \over h_m} - e^{t {\hat \ell^{m+1}}
(p)}{\sin p h_{m+1} \over h_{m+1}} \bigg) e^{i p (y - x)
}\bigg\lvert \notag \\
& + {1 \over 2 L} \bigg\lvert \sum_{p\in B_{m+1}\setminus B_m } e^{t
{\hat \ell^{m+1}} (p)} e^{i p (y - x) }\bigg\lvert.
\notag \\
\end{align}
Let
\begin{equation}
K_m = 2 \sqrt{ \lvert \log h_{m+1} \lvert\over \sigma^2 t}.
\end{equation}
If $h_m$ is small enough, we have that
\begin{align}
&{1 \over 2 L} \bigg\lvert \sum_{p\in B_{m}, \lvert p \lvert \ge K_m
} e^{t {\hat \ell^{m}} (p)} {\sin p h_m \over h_m} e^{i p (y - x)
}\bigg\lvert \leq {1 \over 2 L} \sum_{p\in B_{m}, \lvert p \lvert
\ge K_m }
e^{t \Re (\hat\ell^{m} (p))} {\sin p h_m \over h_m} \notag \\
& \hskip 4cm\leq c \bigg\lvert {\sin K h_m \over h_m} \bigg\lvert
\exp\bigg( t \sigma^2 {\cos K h_m - 1\over h_m^2} \bigg) \leq c
h_m^2. \notag \\
\end{align}
where $c$ denotes a generic constant. Similarly
\begin{align}
{1 \over 2 L} \bigg\lvert \sum_{p\in B_{m+1}, \lvert p \lvert \ge
K_m } e^{t {\hat \ell^{m+1}} (p)} e^{i p (y - x) }\bigg\lvert \leq c
h^2.
\end{align}

If  $m$ is large enough, we also find
\begin{align}
{1 \over 2 L} \bigg\lvert & \sum_{p\in B_m, \lvert p \lvert \leq
K_m} \bigg( {\sin p h_m \over h_m} e^{t {\hat \ell^m} (p)} - {\sin p
h_{m+1} \over h_{m+1}} e^{t {\hat \ell^{m+1}} (p)} \bigg) e^{i p (y
- x) }\bigg\lvert \notag \\
&\leq {1 \over 2 L}  \sum_{p\in B_m, \lvert p \lvert \leq K_m}
\bigg\lvert {\sin p h_{m} \over h_{m}} \bigg\lvert e^{- {1\over4}
{p^2}} \bigg( e^{{\mu t\over 4} h_m^2 \lvert p \lvert^3 + {\sigma^2
t\over16} h_m^2 p^4 } - 1\bigg) \notag
\\ & + e^{- {1\over4} {p^2}} \bigg\lvert
{\sin p h_{m+1} \over h_{m+1}} - {\sin p h_{m} \over h_{m}}
\bigg\lvert \leq c h_m^2 \notag \\
\end{align}
for some constant $c>0$ independent of $m$. This concludes the proof
of the bound of the first derivative. The second derivative can be
derived in a similar way.

Finally, consider the following Fourier representation for the
discretized kernel
\begin{equation}
u_m^{\delta t} (x, y; t) = {1 \over L} \sum_{p\in B_m} \bigg( 1 +
\delta t \hat {\ell^m} (p)\bigg)^{t\over \delta t} e^{i p (y - x) }.
\end{equation}
Consider the formula
\begin{equation}
\bigg( 1 + \delta t \hat {\ell^m} (p)\bigg)^{t\over \delta t} =
\exp\bigg( t \log \big(1 + \hat {\ell^m} (p)\big) \bigg).
\end{equation}
and let's represent the difference between the discrete and
continuous time kernels as follows:
\begin{align}
\lvert &u_m(x, y; t) - u_{m}^{\delta t}(x, y; t)
\lvert \notag \\
& \leq {1 \over 2 L} \bigg\lvert \sum_{p\in B_m} \bigg( \exp\big(t
{\hat \ell^m} (p)\big) - \exp\bigg( {t\over \delta t} \log \big(1 +
\delta t {\hat\ell}^m (p)\big) \bigg) e^{i p (y - x)
}\bigg\lvert \notag \\
& \leq {1 \over 2 L} \sum_{p\in B_m, p\leq K_m} e^{-{1\over4} p^2}
\bigg\lvert \exp\bigg( {t\over \delta t} \log \big(1 + \delta t
{\hat\ell}^m (p)\big) - t {\hat \ell^m} (p) \bigg) -1 \bigg\lvert
\notag \\
&\hskip1cm {1 \over 2 L} \sum_{p\in B_m, p \ge K_m} \bigg\lvert
\exp\big(t {\hat \ell^m} (p)\big)\bigg\lvert  +{1 \over 2 L}
\sum_{p\in B_m, p \ge K_m} \bigg\lvert \exp\bigg({t\over \delta t}
\log \big(1 + \delta t {\hat\ell}^m (p)\big) \bigg) \bigg\lvert
\notag \\
\end{align}
where $K_m$ is chosen as in (\ref{bm_km}). The very same bounds
above lead to the conclusion that this difference is $\leq c h_m^2$.

\section{Smooth Coefficients}

In this section, we prove Theorem 1 in the case where the drift and
volatility are both of class ${\mathcal C}^{(3,0)}$, i.e. they
depend smoothly on the space coordinate but not on the time
coordinate.

Let us introduce the following two constants characterizing the
volatility function:
\begin{equation}
\Sigma_0 = \inf_{x \in A_m} \sigma(x), \;\;\;\;\; \Sigma_1 = \sup_{x
\in A_m } \sqrt{\sigma(x)^2 + h_m \lvert \mu(x) \lvert }.
\end{equation}
and let
\begin{equation}
M = \sup_{x \in A_m} \lvert \mu(x)\lvert .
\end{equation}
Since our interval is bounded, we have that $\Sigma_0>0$ and
$\Sigma_1, M <\infty$.

A {\it symbolic path}  $\gamma = \{\gamma_0, \gamma_1, \gamma_2,
.... \}$ is an infinite sequence of sites in $A_m$ such that
$\gamma_j \neq \gamma_{j-1}$ for all $j=1, ...$. Let $\Gamma_m$ be
the set of all symbolic paths in $A_m$. The kernel of the diffusion
process admits the following representation in terms of a summation
over symbolic paths
\begin{align}
u_m(x, y; t) =& {1\over h_m} \sum_{q=1}^\infty 2 ^{-q} \sum_{
\begin{matrix}
\gamma\in \Gamma_m : \gamma_0 = x, \gamma_q = y \\
\lvert\gamma_j - \gamma_{j-1}\lvert = 1 \;\;\forall j\ge1
\end{matrix}}
 W_m(\gamma, q, t)
 \label{eq_udiff1}
 \end{align}
where
\begin{align}
W_m(&\gamma, q, t) = {1\over h_m} \int_0^{t} ds_1 \int_{s_1}^{t}
ds_2 ... \int_{s_{q-1}}^{t} d s_{q} e^{ (t - s_q) {\mathcal
L}_m(\gamma_q, \gamma_q)} \prod_{j=0}^{q-1}  \bigg( e^{ (s_{j+1} -
s_{j}) {\mathcal L}_m(\gamma_j, \gamma_{j})} 2 {\mathcal
L}_m(\gamma_j, \gamma_{j+1}) \bigg) \label{eq_wdiff}
\end{align}
with $s_0 = 0$.

\begin{figure}
\begin{center}
    \includegraphics[width = 12cm]{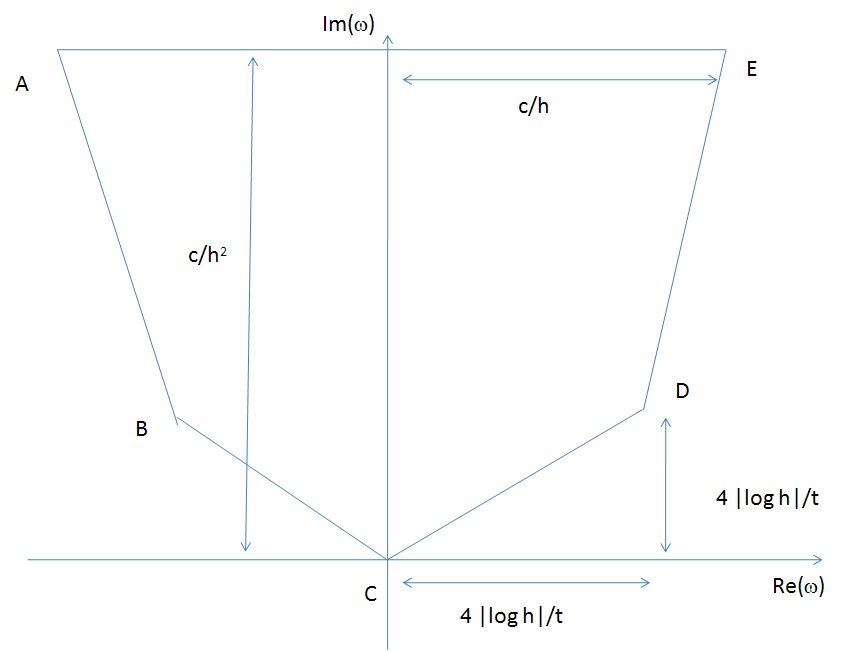}
    \caption{Contour of integration for the integral in (\ref{eq_greenfunc}).
    ${\mathcal C}_+$ is the countour joining the point $D$ to the
    points $E, A, B$. ${\mathcal C}_-$ is
    the countour joining the point $B$ to $C$ to $D$.}
    \label{fig_contour2}
\end{center}
\end{figure}

Let us introduce the following Green's function:
\begin{equation}
G_{m}(x, y; \omega) = \int_0^\infty u_{m}(x, y; t) e^{-i\omega t} dt
= h_m^{-1} {1\over {\mathcal L}^m + i\omega}(x, y).
\end{equation}
The propagator can be expressed as the following contour integral
\begin{equation}
u_{m}(x, y; t) = \int_{{\mathcal C}_-} {d\omega \over 2 \pi}
G_{m}(x, y; \omega) e^{i\omega t} + \int_{{\mathcal C}_+} {d\omega
\over 2 \pi} G_{m}(x, y; \omega) e^{i\omega t}. \label{eq_greenfunc}
\end{equation}
Here, ${\mathcal C}_+$ is the contour joining the point $D$ to the
points $E, A, B$ in Fig. \ref{fig_contour2}, while ${\mathcal C}_-$
is the contour joining the point $B$ to $C$ to $D$. By design, each
point $\omega$ on the upper path ${\mathcal C}_+$ is separated from
the spectrum of $\mathcal L$.

\begin{lemma} \label{lemma_cplus} For $m$ sufficiently large,
there is a constant $c>0$ such that
\begin{equation}
\bigg\lvert\int_{{\mathcal C}_+} {d\omega \over 2 \pi} G_{m}(x, y;
\omega) e^{i\omega t}\bigg\lvert \leq c h^2 .
\end{equation}
\end{lemma}
\begin{proof}
The proof is based on the geometric series expansion
\begin{equation}
G_{m}(\omega) = h_m^{-1} {1\over {\mathcal L}^m + i\omega} =
h_m^{-1} \sum_{j=0}^\infty {1\over {1\over2}\sigma^2\Delta^m +
i\omega} \bigg[ \mu\nabla^m {1\over {1\over2}\sigma^2\Delta^m +
i\omega} \bigg]^j \label{geoseries}
\end{equation}
whose convergence for $\omega\in{\mathcal C}_+$ can be established
by means of a Kato-Rellich relative bound, see \cite{Kato}. More
precisely, for any $\alpha>0$, one can find a $\beta>0$ such that
the operators $\nabla^m$ and $\Delta^m$ satisfy the following
relative bound estimate:
\begin{equation}
\lvert\lvert \nabla^m f\lvert\lvert_2 \leq \alpha \lvert\lvert
\Delta^m f \lvert\lvert_2 + \beta \lvert\lvert f \lvert\lvert_2.
\end{equation}
for all periodic functions $f$ and all $m\ge m_0$. This bound can be
derived by observing that $\nabla^m$ and $\Delta^m$ can be
diagonalized simultaneously by a Fourier transform, as done in the
previous section, and by observing that for any $\alpha>0$, one can
find a $\beta>0$ such that
\begin{equation}
\bigg\lvert {\sin h_m p \over h_m}  \bigg\lvert \leq \alpha
\bigg\lvert {\cos h_m p - 1 \over h_m^2} \bigg\lvert + \beta
\end{equation}
for all $m\ge m_0$ and all $p\in B_m$.

Under the same conditions, we also have that
\begin{equation} \big\lvert\big\lvert
\mu \nabla^m f\big\lvert\big\lvert_2 \leq {2 M \alpha \over
\Sigma_0^2} \bigg\lvert\bigg\lvert {1\over2} \sigma^2 \Delta^m f
\bigg\lvert\bigg\lvert_2 + \beta \lvert\lvert f \lvert\lvert_2.
\end{equation}
Hence
\begin{equation} \bigg\lvert\bigg\lvert
\mu \nabla^m {1\over {1\over2}\sigma^2\Delta^m + i\omega}
f\bigg\lvert\bigg\lvert_2 \leq {2 M \alpha \over \Sigma_0^2}
\bigg\lvert\bigg\lvert {1\over2} \sigma^2 \Delta^m {1\over
{1\over2}\sigma^2\Delta^m + i\omega} f \bigg\lvert\bigg\lvert_2 +
\beta \bigg\lvert\bigg\lvert {1\over {1\over2}\sigma^2\Delta^m +
i\omega} f \bigg\lvert\bigg\lvert_2 < 1
\end{equation}
where the last inequality holds if $\omega\in{\mathcal C}_+$, if
$\alpha$ is chosen sufficiently small and if $m$ is large enough. In
this case, the geometric series expansion converges in
(\ref{geoseries}) converges in $L^2$ operator norm. The uniform norm
of the kernel $\lvert G_{m}(x, y; \omega) \lvert$ is pointwise
bounded from above by $h_m^{-1}$.

Since the points $B$ and $D$ have imaginary part equal at height $
4{\lvert \log h_m \lvert \over t}$, the integral over the contour
${\mathcal C}_+$ converges also and is bounded from above by $c
h_m^2$ in uniform norm.

\end{proof}

\begin{lemma} If $q\ge {e^2
\Sigma_1^2 t\over 2 h_m^2}$ we have that
\begin{equation}
W_m(\gamma, q; t) \leq \sqrt{q\over 2\pi} \exp\left(-{\Sigma_0^2
t\over 2} - q\right). \label{eq_wbound}
\end{equation}
\end{lemma}
\begin{proof}
Let us define the function
\begin{equation}
\phi(t) = {\Sigma_1^2\over 2 h_m^2} \; e^{-{\Sigma_0^2 t \over 2
h_m^2}} \; 1(t\ge 0)
\end{equation}
where $1(t\ge0)$ is the characteristic function of $\RRR_+$. We have
that
\begin{equation}
W_m(\gamma, q; t) \leq \phi^{\star q}(t)
\end{equation}
where $\phi^{\star q}$ is the $q-$th convolution power, i.e. the
$q-$fold convolution product of the function $\phi$ by itself. The
Fourier transform of $\phi(t)$ is given by
\begin{equation}
\hat\phi(\omega) = {\Sigma_1^2\over 2 h_m^2} \int_0^\infty
e^{-i\omega t - {\Sigma_0^2 t\over 2 h_m^2}} dt = {\Sigma_1^2 \over
2 i \omega h_m^2 + \Sigma_0^2}.
\end{equation}
The convolution power is given by the following inverse Fourier
transform:
\begin{equation}
\phi^{\star q}(t) = \int_0^\infty \hat \phi(\omega)^q e^{i \omega t}
= \left( {\Sigma_1\over\Sigma_0} \right)^{2 q}
\int_{-\infty}^\infty\left( 1 + {2 i \omega h_m^2 \over \Sigma_0^2}
\right)^{-q} e^{i \omega t} {d\omega\over 2\pi}.
\end{equation}
Introducing the new variable $z = 1 + {2i\omega h_m^2\over
\Sigma_0^2}$, the integral can be recast as follows
\begin{equation}
\phi^{\star q}(t) = {\Sigma_0^{2-2q} \Sigma_1^{2q}\over 4\pi i
h_m^2} \lim_{R\to\infty} \int_{{\mathcal C}_R} z^{-q} \exp\left(
{\Sigma_0^2 t\over 2 h_m^2} (z-1) \right) dz \label{eq_intc}
\end{equation}
where ${\mathcal C}_R$ is the contour in Fig. \ref{fig_contour1}.
Using the residue theorem and noticing that the only pole of the
integrand is at $z = 0$, we find
\begin{equation}
\phi^{\star q}(t) = {1\over (q-1)!} \left({\Sigma_1^2 t \over 2
h_m^2}\right)^q \exp\left( - \Sigma_0^2 t \over 2 h_m^2 \right).
\end{equation}
Making use of Stirling's formula $q! \approx \sqrt{2\pi}
q^{q+{1\over2}} e^{-q}$, we find
\begin{equation}
\phi^{\star q}(t) \approx \sqrt{q\over 2\pi} \exp\left( -
{\Sigma_0^2 t \over 2 h_m^2} + q\log{\Sigma_1^2 t \over 2 h_m^2} + q
(1-\log q) \right).
\end{equation}
If $\log q \ge \log {\Sigma_1^2 t \over 2 h_m^2} +2$, then we arrive
at the bound in (\ref{eq_wbound}).

\begin{figure}
\begin{center}
    \includegraphics[width = 12cm]{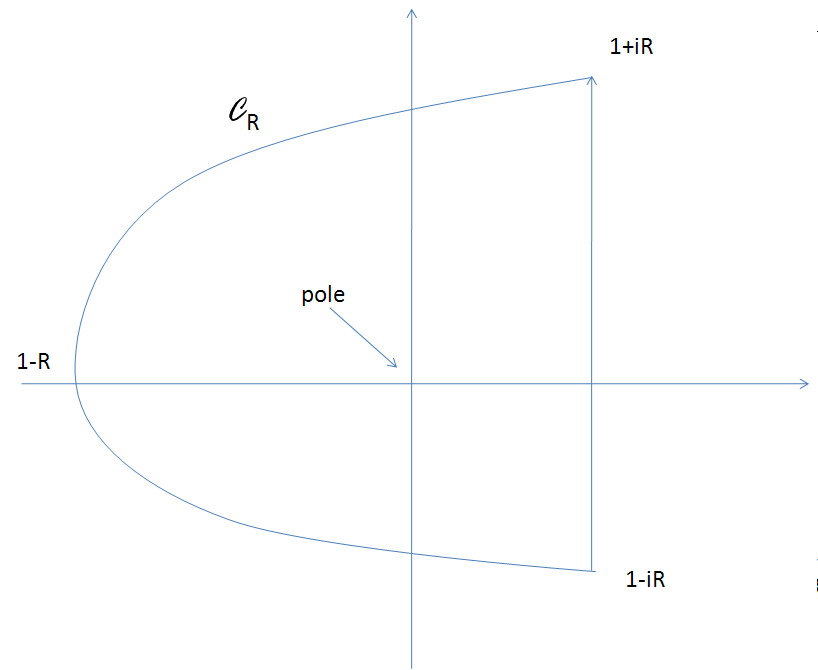}
    \caption{Contour of integration ${\mathcal C}_R$ for the integral in (\ref{eq_intc}).}
    \label{fig_contour1}
\end{center}
\end{figure}

\end{proof}

It suffices to consider the case $m' = m+1$ for all values of $m$
above a fixed threshold. In fact, given this particular case, the
general statement can be derived with an iterative argument. To this
end, we introduce a renormalization group transformation based on
the notion of decorating path.

\begin{definition}{\bf (Decorating Paths.)}
 Let $m\ge m_0$ and let $\gamma = \{y_0, y_1, y_2,
.... \}$ be a symbolic sequence in $\Gamma_m$. A {\it decorating
path around $\gamma$} is defined as a symbolic sequence $\gamma' =
\{y_0, y_1', y_2', .... \}$ with $y_i' \in h_{m+1} \ZZZ$ containing
the sequence $\gamma$ as a subset and such that if $y_j' = y_i$ and
$y_k' = y_{i+1}$, then all elements $y_n'$ with $j < n < k$ are such
that $\lvert y_n' - y_j' \lvert \leq h_{m+1}$. Let ${\mathcal
D}_{m+1}(\gamma)$ be the set of all decorating sequences around
$\gamma$. The decorated weights are defined as follows:
\begin{equation}
{\tilde W}_{m}(\gamma, q; t) = \sum_{q' = q}^\infty \sum_{
\begin{matrix}
\gamma' \in {\mathcal D}_{m+1}(\gamma)\\
\gamma'_{q'} = \gamma_q
\end{matrix}
} W_{m+1}(\gamma', q'; t).
\end{equation}
Finally, let us introduce also the following Fourier transform:
\begin{equation}
\hat W_{m}(\gamma, q; \omega) = \int_0^\infty W_{m}(\gamma, q; 0, t)
e^{i\omega t} dt, \;\;\;\; \hat {\tilde W}_{m}(\gamma, q; \omega) =
\int_0^\infty \tilde W_{m}(\gamma, q; t) e^{i\omega t} dt.
\end{equation}
\end{definition}

\begin{definition}{\bf (Notations.)}
In the following, we set $h = h_{m+1}$ so that $h_m = 2 h$. We also
use the Landau notation $O(h^n)$ to indicate a function $f(h)$ such
that $ h^{-n} f(h) $ is bounded in a neighborhood of $(0)$.
\end{definition}

\begin{lemma}  Let $x, y\in A_m$ and let ${\mathcal C_-}$ be an
integration contour as in Fig. \ref{fig_contour2}. Then
\begin{equation}
\bigg\lvert \bigg( \int_{{\mathcal C}_-}  2 G_{m+1}(x, y; \omega) -
G_m(x, y; \omega) \bigg) e^{i\omega t} {d\omega\over 2\pi}
\bigg\lvert = O(h^2).
\end{equation}
\end{lemma}

\begin{proof}

We have that
\begin{align}
2 G_{m+1}(x, y; \omega) - G_m(x, y; \omega) = & {1\over h}
\sum_{q=1}^\infty 2 ^{-q} \sum_{
\begin{matrix}
\gamma\in \Gamma_m : \gamma_0 = x, \gamma_{q} = y\\
\lvert\gamma_j - \gamma_{j-1}\lvert = 1 \forall j\ge1
\end{matrix}}
\left( 2 \hat {\tilde W}_m (\gamma, q; \omega) - \hat W_m (\gamma,
q; \omega)\right). \label{eq_udiff}
\end{align}
The number of paths over which the summation is extended is
\begin{equation}
N(\gamma, q; x, y) \equiv \sharp\{\gamma\in \Gamma_m : \gamma_0 = x,
\gamma_{q} = y, \lvert\gamma_j - \gamma_{j-1}\lvert = 1 \forall
j\ge1 \} = \left(\begin{matrix} q \\ {q\over2} + k
\end{matrix}\right)
\end{equation}
where $k = {\lvert y - x \lvert \over h_m}.$ Applying Stirling's
formula we find
\begin{equation}
N_\gamma \lesssim 2^q \sqrt{2\over \pi q}.
\end{equation}
Hence
\begin{align}
&\bigg\lvert \int_{{\mathcal C}_-} \bigg( 2 G_{m+1}(x, y; \omega) -
G_m(x, y; \omega) \bigg) e^{i\omega t}  {d\omega\over 2\pi}
\bigg \lvert \notag \\
&\leq {c\over h}
 \sum_{q=1}^\infty \sqrt{1\over q} \max_{
\begin{matrix}
\gamma\in \Gamma_m : \gamma_0 = x, \gamma_{q} = y\\
\lvert\gamma_j - \gamma_{j-1}\lvert = 1 \forall j\ge1
\end{matrix}}
\bigg \lvert \int_{{\mathcal C}_-} \bigg(2 \hat {\tilde W}_m
(\gamma, q; \omega) - \hat W_m (\gamma, q; \omega)\bigg) e^{i\omega
t}  {d\omega\over 2\pi} \bigg\lvert. \notag
\\\label{eq_udiff}
\end{align}
for some constant $c\approx\sqrt{2\over\pi} >0$. It suffices to
extend the summation over $q$ only up to
\begin{equation}
q_{\rm max} \equiv {e^2 \Sigma_1^2 t\over 2 h^2}.
\end{equation}
To resum beyond this threshold, one can use the previous lemma. More
precisely, we have that
\begin{align}
&\bigg\lvert \int_{{\mathcal C}_-} \bigg( 2 G_{m+1}(x, y; \omega) -
G_m(x,
y; \omega) \bigg) e^{i\omega t}  {d\omega\over 2\pi} \bigg\lvert \notag \\
&\leq {c \sqrt{q_{\rm max}}\over h} \max_{
\begin{matrix}
q, \gamma\in \Gamma_m : \gamma_0 = x, \gamma_{q} = y\\
\lvert\gamma_j - \gamma_{j-1}\lvert = 1 \forall j\ge1
\end{matrix}}
\bigg \lvert \int_{{\mathcal C}_-} \bigg( 2 \hat {\tilde W}_m
(\gamma, q; \omega) - \hat W_m (\gamma, q; \omega) \bigg) e^{i\omega
t}  {d\omega\over 2\pi}  \bigg\lvert. \notag \\\label{eq_udiff2}
\end{align}

Let $v(x) = \sigma(x)^2$. To evaluate the resummed weight function,
let us form the matrix
\begin{equation} \bar {\mathcal L}(x;h) =
\left(
\begin{matrix}
-{v\left(x+h\right)\over h^2} && {v\left(x+ h\right)\over 2 h^2} -
{\mu(x+h)\over 2 h} && 0 \\
{v\left(x\right)\over 2 h^2}  +{\mu(x)\over 2h} && -{ v\left(x \right)\over h^2}  && {v\left(x\right) \over 2 h^2} - {\mu(x) \over 2 h} \\
0 && {v\left(x - h\right)\over 2 h^2} + {\mu(x-h)\over 2 h} &&
-{v\left(x - h \right) \over h^2}
\end{matrix} \right)
\end{equation}
and decompose it as follows:
\begin{equation}
\bar {\mathcal L}(x;h) = {1\over h^2} \bar {\mathcal L}_0(x) +
{1\over h} \bar {\mathcal L}_1(x) + \bar {\mathcal L}_2(x) + h \bar
{\mathcal L}_3(x) + O(h^2).
\end{equation}
where
\begin{equation}
\bar {\mathcal L}_0(x) = \left(
\begin{matrix}
-v(x) && {1\over2} {v(x)} && 0 \\
{1\over2} {v(x)} && -{v(x)} && {1\over2} {v(x)} \\
0 && {1\over2} {v(x)} && -{v(x)}
\end{matrix}
\right),
\end{equation}
\begin{equation}
\bar {\mathcal L}_1(x) = \left(
\begin{matrix}
-v'(x) && {1 \over2} v'(x) - {1\over2} \mu(x) && 0 \\
{1\over2} \mu(x) && 0 && -{1\over2} \mu(x) \\
0 && - {1\over2} v'(x) + {1\over2}  \mu(x) &&  v'(x)
\end{matrix}\right),
\end{equation}
\begin{equation}
\bar {\mathcal L}_2(x) = \left(
\begin{matrix}
-{1\over2} v''(x)  && {1\over4} v''(x) - {1\over2} \mu'(x)&& 0 \\
0 && 0 && 0 \\
0 && {1\over4} v''(x) - {1\over2} \mu'(x) && -{1\over2} v''(x)
\end{matrix}\right).
\end{equation}
and
\begin{equation}
\bar {\mathcal L}_3(x) = \left(
\begin{matrix}
-{1\over6} v'''(x)  && {1\over12} v'''(x) - {1\over4} \mu''(x)&& 0 \\
0 && 0 && 0 \\
0 && - {1\over12} v'''(x) + {1\over4} \mu''(x) && {1\over6} v'''(x)
\end{matrix}\right).
\end{equation}

Let us introduce the sign variable $\tau = \pm1$, the functions
\begin{align}
\phi_{0}(t, x, \tau) &\equiv 2 {\mathcal L}_m(x, x+2\tau h) e^{ t
{\mathcal
L}_m(x, x)} 1(t\ge0) \\
\phi_{1}(t, x, \tau) &\equiv 2 {\mathcal L}_{m+1}(x+\tau h, x+2\tau
h) e^{t \bar {\mathcal L}(x; h)}(x, x+\tau h) 1(t\ge0)
\end{align}
and their Fourier transforms
\begin{align}
\hat\phi_{0}(\omega, x, \tau) &= \left({v(x)\over 4 h^2} + \tau {
\mu(x)\over 2 h}\right) \left( {v(x)\over 4 h^2} + i \omega
\right)^{-1} \notag \\ \hat\phi_{1}(\omega, x, \tau) &=
\left({v(x)\over h^2} + \tau  {\mu(x) + v'(x)\over h} +
{v''(x)+\mu'(x)\over2} +\left( {v'''(x)\over6} + {\mu''(x)\over2}
\right)\tau h +
O(h^2)\right) \notag \\
&\hskip8cm < x \lvert \left( -\bar {\mathcal L}(x; h) + i\omega
\right)^{-1}\lvert x+\tau h >.
\end{align}
where
\begin{equation}
\lvert x> = \left(\begin{matrix} 0 \\ 1 \\
0 \end{matrix}\right), \;\;\;\;{\rm and}\;\;\;\; \lvert x+\tau h> =
\left( \begin{matrix} \delta_{\tau, 1} \\ 0 \\
\delta_{\tau, -1} \end{matrix}\right).
\end{equation}
We also require the functions
\begin{equation}
\psi_{0}(t, x) \equiv e^{ t {\mathcal L}_m(x, x) } 1(t\ge0),
\;\;\;\;\;\; \psi_{1}(t, x) \equiv e^{t \bar {\mathcal L}(y; h)}(x,
x) 1(t\ge0)
\end{equation}
and the corresponding Fourier transforms
\begin{align}
\hat\psi_{0}(\omega, x) = \left( {v(x)\over 4 h^2}  + i \omega
\right)^{-1}, \;\;\;\;\; \hat\psi_{1}(\omega, x) =  <x \lvert \left(
-\bar {\mathcal L}(x; h) + i\omega \right)^{-1}\lvert x>.
\end{align}

If $\gamma$ is a symbolic sequence, then
\begin{align}
\hat W_m(\gamma, q; \omega) &= \hat\psi_{0}(\omega, \gamma_q)
\prod_{j=0}^{q-1} \hat\phi_0(\omega;
\gamma_j, {\rm sgn}(\gamma_{j+1}-\gamma_j))\\
 \hat {\tilde W}_m(\gamma, q; \omega)
&= \hat\psi_{1}(\omega, \gamma_q) \prod_{j=0}^{q-1}
\hat\phi_1(\omega; \gamma_j, {\rm sgn}(\gamma_{j+1}-\gamma_j)).
\end{align}

Let us estimate the difference between the functions
$\hat\phi_{1}(\omega, x, \tau)$ and $\hat\phi_{2}(\omega, x, \tau)$
assuming that $\omega$ is in the contour ${\mathcal C}_-$ in Fig.
\ref{fig_contour1}. Retaining only terms up to order up to $O(h^3)$,
we find
\begin{eqnarray}
\hat\phi_{0}(\omega, x, \tau) = 1 + {2\mu(x)\tau h \over v(x)} - {4
i \omega h^2 \over v(x)} - 8\mu(x) {i\omega\tau h^3 \over v(x)^2}
-{16\omega^2 h^4 \over v(x)^2}
+ O(h^5). \notag \\
\end{eqnarray}
A lengthy but straightforward calculation which is best carried out
using a symbolic manipulation program, gives
\begin{align}
&\hat\phi_{1}(\omega, x, \tau) = 1+{2\mu(x) \tau h \over v(x)}-{ 4 i
\omega h^2 \over v(x)}  - \big[8\mu(x)  - v'(x) \big]
{i\omega\tau h^3 \over v(x)^2} \notag \\
&\hskip6cm+ r(x)\cdot h^3 \tau
 + i\omega h^4 p(x) - {14\omega^2 h^4 \over v(x)^2} + O(h^5) \notag \\
\end{align}
where
\begin{align}
& r(x) = {1\over 2 v(x)^3} \big[ \mu''(x) v(x) -  4\mu(x)^3 + 2
v'(x) \mu(x)^2 - 2v'(x) v(x) \mu'(x)  \notag \\
&\hskip6cm  - \big(2\mu(x) \mu'(x) + v''(x) v(x) - 2v'(x)^2\big) \mu(x) \big] . \notag \\
&p(x) = {1\over v(x)^3} \big[4\mu(x)^2-2v'(x)\mu(x)+4v(x)\mu'(x)+
v''(x)
v(x)- 2 v'(x)^2  \big].  \notag \\
\end{align}

We have that
\begin{align}
& \sum_{j=0}^{q-1} \bigg(\log \hat\phi_0(\omega; \gamma_j, {\rm
sgn}(\gamma_{j+1}-\gamma_j)) -
 \log \hat\phi_1(\omega; \gamma_j, {\rm
sgn}(\gamma_{j+1}-\gamma_j))\bigg)
\notag \\
&= \sum_{j=0}^{q-1} \bigg(  {i \omega v'(\gamma_j)\over
v(\gamma_j)^2} + r(\gamma_j) \bigg) h^3 {\rm
sgn}(\gamma_{j+1}-\gamma_j) + \big(\lvert \omega\lvert \lvert\lvert
p\lvert\lvert_{\infty} + 2 \lvert \omega\lvert^2
\lvert\lvert v^{-2} \lvert\lvert_{\infty} \big) O(h^4 q) \notag \\
&= i\omega h^2 \log\bigg({v(\gamma_{q})\over v(\gamma_0)}\bigg) +
h^2 \big(R(\gamma_q) - R(\gamma_0)\big) + \big(\lvert \omega\lvert
\lvert\lvert p\lvert\lvert_{\infty} + 2 \lvert \omega\lvert^2
\lvert\lvert v^{-2} \lvert\lvert_{\infty} \big)
O( h^4 q) \notag \\
\end{align}
where $R(x)$ is a primitive of $r(x)$, i.e.
\begin{equation}
R(x) = \int^x r(z) dz.
\end{equation}

We conclude that there is a constant $c>0$ such that
\begin{equation}
\bigg\lvert \int_{{\mathcal C}_-} \bigg( \prod_{j=0}^{q-1}
\hat\phi_0(\omega; \gamma_j, {\rm sgn}(\gamma_{j+1}-\gamma_j)) -
\prod_{j=0}^{q-1} \hat\phi_1(\omega; \gamma_j, {\rm
sgn}(\gamma_{j+1}-\gamma_j)) \bigg) e^{i\omega t}  {d\omega\over
2\pi} \bigg\lvert \leq c h^2.
\end{equation}
for all $q\leq q_{\max}$. Here we use the decay of $e^{i\omega t}$
in the upper half of the complex $\omega$ plane to offset the
$\omega$ dependencies in the integrand. Similar calculations lead to
the following expansions:
\begin{equation}
\hat\psi_{0}(\omega, y) = {4 h^2 \over v(y)}  + O(\omega h^4),
\;\;\;\;\; \hat\psi_{1}(\omega, y) =  {2 h^2 \over v(y)} + O(\omega
h^4) = {1\over 2} \hat\psi_{0}(\omega, y) + O(\omega h^4).
\end{equation}

Since $q<c h^{-2}$ and $\omega \leq \lvert \log h \lvert$, we find
\begin{align}
\bigg\lvert \int_{{\mathcal C}_-} \bigg(  2 G_{m+1}(x, y; \omega) -
G_m(x, y; \omega) \bigg) e^{i\omega t}  {d\omega\over 2\pi}
\bigg\lvert \leq c { \sqrt{q_{\rm max}} \over h }  h^4 \leq c h^2.
\label{eq_udiff3}
\end{align}
This completes the proof of the Lemma and of the Theorem.
\end{proof}

By differentiating with respect to time in equation
\ref{eq_greenfunc}, we find that
\begin{equation}
{\partial \over \partial t} u_{m}(x, y; t) = \int_{{\mathcal C}_-}
{d\omega \over 2 \pi} i\omega G_{m}(x, y; \omega) e^{i\omega t} +
\int_{{\mathcal C}_+} {d\omega \over 2 \pi} i\omega  G_{m}(x, y;
\omega) e^{i\omega t}. \label{eq_greenfunc}
\end{equation}
All the derivations above carry through and we conclude that
\begin{equation}
\bigg\lvert\int_{{\mathcal C}_+} {d\omega \over 2 \pi} i\omega
G_{m}(x, y; \omega) e^{i\omega t}\bigg\lvert \leq c h^2 .
\end{equation}
and also
\begin{align}
\bigg\lvert \int_{{\mathcal C}_-} i\omega \bigg(  2 G_{m+1}(x, y;
\omega) - G_m(x, y; \omega) \bigg) e^{i\omega t}  {d\omega\over
2\pi} \bigg\lvert \leq c { \sqrt{q_{\rm max}} \over h }  h^4 \leq c
h^2. \label{eq_udiff3}
\end{align}
Hence the first time derivatives of the kernel satisfy the same
Cauchy convergence condition as the kernel itself.

\section{Lesser Smooth Coefficients}

In this section we assume coefficients are either H\"older
continuous or obey the conditions in Theorem 2.

\begin{lemma} Let $f(x)$ be a continuous function in $[-L, L]$
satisfying periodic boundary conditions. Then, for all $h>0$, we
have that
\begin{align}
f(x+h_m) = f(x) + h_m \nabla^m_x f(x) + {h^2\over 2} \Delta^m_x
f(x)\\
f(x-h_m) = f(x) - h_m \nabla^m_x f(x) + {h^2\over 2} \Delta^m_x
f(x).
\end{align}
\end{lemma}
This is the result of a simple calculation, which is however useful
as it allows one to extend the derivation in the previous section by
making the following replacements:
\begin{align}
v'(x) \to \nabla^m_x v(x), \;\;\;\; v''(x) \to \Delta^m v(x), \;\;\;\; v'''(x) \to 0 \\
\mu'(x) \to \nabla^m_x \mu(x), \;\;\;\; \mu''(x) \to \Delta^m_x
\mu(x).
\end{align}
In fact,
\begin{equation}
\bar {\mathcal L}(x;h) = {1\over h^2} \bar {\mathcal L}_0(x) +
{1\over h} \bar {\mathcal L}_1(x) + \bar {\mathcal L}_2(x) + h \bar
{\mathcal L}_3(x)
\end{equation}
without any $O(h^3)$ corrections as long as one re-defines the
matrices on the right hand side as follows:
\begin{equation} \bar {\mathcal
L}_0(x) = \left(
\begin{matrix}
-v(x) && {1\over2} {v(x)} && 0 \\
{1\over2} {v(x)} && -{v(x)} && {1\over2} {v(x)} \\
0 && {1\over2} {v(x)} && -{v(x)}
\end{matrix}
\right),
\end{equation}
\begin{equation}
\bar {\mathcal L}_1(x) = \left(
\begin{matrix}
-\nabla_x^m v(x) && {1 \over2} \nabla_x^m v(x) - {1\over2} \mu(x) && 0 \\
{1\over2} \mu(x) && 0 && -{1\over2} \mu(x) \\
0 && - {1\over2} \nabla_x^m v(x) + {1\over2}  \mu(x) && \nabla_x^m
v(x)
\end{matrix}\right),
\end{equation}
\begin{equation}
\bar {\mathcal L}_2(x) = \left(
\begin{matrix}
-{1\over2} \Delta_x^m v(x)  && {1\over4} \Delta_x^m v(x) - {1\over2} \nabla_x^m \mu(x)&& 0 \\
0 && 0 && 0 \\
0 && {1\over4} \Delta_x^m v(x) - {1\over2} \nabla_x^m \mu(x) &&
-{1\over2} \Delta_x^m v(x)
\end{matrix}\right).
\end{equation}
and
\begin{equation}
\bar {\mathcal L}_3(x) = \left(
\begin{matrix}
0  && - {1\over4} \Delta_x^m \mu(x)&& 0 \\
0 && 0 && 0 \\
0 && {1\over4} \Delta_x^m \mu(x) && 0
\end{matrix}\right).
\end{equation}
All derivations in the previous section go through formally
unchanged and one arrives at the following expressions
\begin{eqnarray}
\hat\phi_{0}(\omega, x, \tau) = 1 + {2\mu(x)\tau h \over v(x)} - {4
i \omega h^2 \over v(x)} - 8\mu(x) {i\omega\tau h^3 \over v(x)^2}
-{16\omega^2 h^4 \over v(x)^2}
+ O(h^5). \notag \\
\end{eqnarray}
and
\begin{align}
&\hat\phi_{1}(\omega, x, \tau) = 1+{2\mu(x) \tau h \over v(x)}-{ 4 i
\omega h^2 \over v(x)}  - \big[8\mu(x)  - \nabla^m_x v(x) \big]
{i\omega\tau h^3 \over v(x)^2} \notag \\
&\hskip6cm+ r(x)\cdot h^3 \tau
 + i\omega h^4 p(x) - {14\omega^2 h^4 \over v(x)^2} + O(h^5) \notag \\
\end{align}
where
\begin{align}
& r(x) = {1\over 2 v(x)^3} \big[ \Delta^m_x \mu(x) v(x) - 4\mu(x)^3
+ 2
\nabla^m_x v(x) \mu(x)^2 - 2 \nabla^m_x v(x) v(x) \nabla^m_x \mu(x)  \notag \\
&\hskip6cm  - \big(2\mu(x) \nabla^m_x \mu(x) + \Delta^m_x v(x) v(x) - 2 \nabla^m_x v(x)^2\big) \mu(x) \big] . \notag \\
&p(x) = {1\over v(x)^3} \big[4\mu(x)^2-2 \nabla^m_x v(x)\mu(x)+4v(x)
\nabla^m_x \mu(x)+ \Delta^m_x v(x)
v(x)- 2 \nabla^m_x v(x)^2  \big].  \notag \\
\end{align}

We have that
\begin{align}
& \bigg\lvert \sum_{j=0}^{q-1} \bigg(\log \hat\phi_0(\omega;
\gamma_j, {\rm sgn}(\gamma_{j+1}-\gamma_j)) -
 \log \hat\phi_1(\omega; \gamma_j, {\rm
sgn}(\gamma_{j+1}-\gamma_j))\bigg)\bigg\lvert
\notag \\
&\leq \bigg\lvert  \sum_{j=0}^{q-1} \bigg(  {i \omega
\nabla^m_{\gamma_j} v(\gamma_j)\over v(\gamma_j)^2} + r(\gamma_j)
\bigg) h^3 {\rm sgn}(\gamma_{j+1}-\gamma_j)\bigg\lvert + \big(\lvert
\omega\lvert \lvert\lvert p\lvert\lvert_{\infty} + 2 \lvert
\omega\lvert^2
\lvert\lvert v^{-2} \lvert\lvert_{\infty} \big) O(h^4 q)  .\notag \\
&\leq 2L h^2 \sup_{x\in A_m} \bigg\lvert   {i\omega \nabla^m_{x}
v(x)\over v(x)^2} + r(x) \bigg \lvert  + \big(\lvert \omega\lvert
\lvert\lvert p\lvert\lvert_{\infty} + 2 \lvert \omega\lvert^2
\lvert\lvert v^{-2} \lvert\lvert_{\infty} \big) O(h^4 q) \leq c h^{\gamma} .\notag \\
\label{eq_413}
\end{align}
where in the last step we made use of the H\"older continuity
assumptions of Theorem 1. The other bounds staying the same, we
arrive at
\begin{align}
\bigg\lvert \int_{{\mathcal C}_-} \bigg(  2 G_{m+1}(x, y; \omega) -
G_m(x, y; \omega) \bigg) e^{i\omega t}  {d\omega\over 2\pi}
\bigg\lvert \leq c { \sqrt{q_{\rm max}} \over h }  h^{2+\gamma} \leq
c h^\gamma. \label{eq_udiff3}
\end{align}
Under the weaker assumption of Theorem 2, the bound that applies is
instead
\begin{align}
\bigg\lvert \int_{{\mathcal C}_-} \bigg(  2 G_{m+1}(x, y; \omega) -
G_m(x, y; \omega) \bigg) e^{i\omega t}  {d\omega\over 2\pi}
\bigg\lvert \leq c { \sqrt{q_{\rm max}} \over h }  h^{2} \rho(h)
\leq c \rho(h). \label{eq_udiff3}
\end{align}
Similar bounds also extend to the case of the first time derivative,
since multiplication by a factor $i\omega$ inside of the contour
integral is immaterial as far as establishing a bound of this sort
is concerned. This completes the proof of Theorem 1 and Theorem 2.

\section{Explicit Euler Scheme}

In this section we prove Theorem \ref{theo_diffdtker}. A Dyson
expansion can also be obtained for the time-discretized kernel and
has the form
\begin{align}
u_m^{\delta t}(y_1, y_2; t) =& {1\over h_m} \sum_{q=1}^\infty
\sum_{\gamma\in \Gamma_m : \gamma_0 = x, \gamma_{q} = y}
 \sum_{k_1 = 1}^{N} \sum_{k_2 = k_1 + 1}^{N} ...
\sum_{k_{q} = k_{q-1} + 1}^{N}  \notag \\
& \bigg(1 + \delta t \LLL_m(\gamma_0, \gamma_0)\bigg)^{k_{1}-1}
(\delta t)^q \prod_{j=1}^{q} \LLL_m(\gamma_{j-1}, \gamma_j) \bigg(1
+ \delta t \LLL_m(\gamma_j, \gamma_j)\bigg)^{k_{j+1} - k_{j} - 1}
\label{eq_upathint}
\end{align}
where $t_{q+1} = t$ and $k_{q+1}=N$. In this case, the propagator
can be expressed through a Fourier integral as follows:
\begin{equation}
u_m^{\delta t}(y_1, y_2; t) = \int_{-{\pi\over\delta
t}}^{\pi\over\delta t} G_m^{\delta t}(y_1, y_2; \omega) e^{i\omega
t} {d\omega\over 2\pi}
\end{equation}
where
\begin{equation}
G_m^{\delta t}(y_1, y_2; \omega) = \delta t \sum_{j = 0}^{t\over
\delta t} u_m^{\delta t}(y_1, y_2; j \delta t) e^{ - i\omega j
\delta t}.
\end{equation}
The propagator can also be represented as the limit
\begin{equation}
u_m^{\delta t}(y_1, y_2; t) = \lim_{H\to\infty} \int_{{\mathcal
C}_H} G_m^{\delta t}(y_1, y_2; \omega) e^{i\omega t} {d\omega\over
2\pi} \label{eq_greenfuncdt}
\end{equation}
where ${\mathcal C}_H$ is the contour in Fig. \ref{fig_contour3}.
This is due to the fact that the integral along the segments $BC$
and $DA$ are the negative of each other, while the integral over
$CD$ tends to zero exponentially fast as $\Im(\omega) \to \infty$,
where $\Im(\omega)$ is the imaginary part of $\omega$. Using
Cauchy's theorem, the contour in Fig. \ref{fig_contour3} can be
deformed into the contour in Fig. \ref{fig_contour2}. To estimate
the discrepancy between the time-discretized kernel and the
continuous time one, one can thus compare the Green's function along
such contour. Again, the only arc that requires detailed attention
is the arc $BCD$, as the integral over rest of the contour of
integration can be bounded from above as in the previous section.

\begin{figure}
\begin{center}
    \includegraphics[width = 12cm]{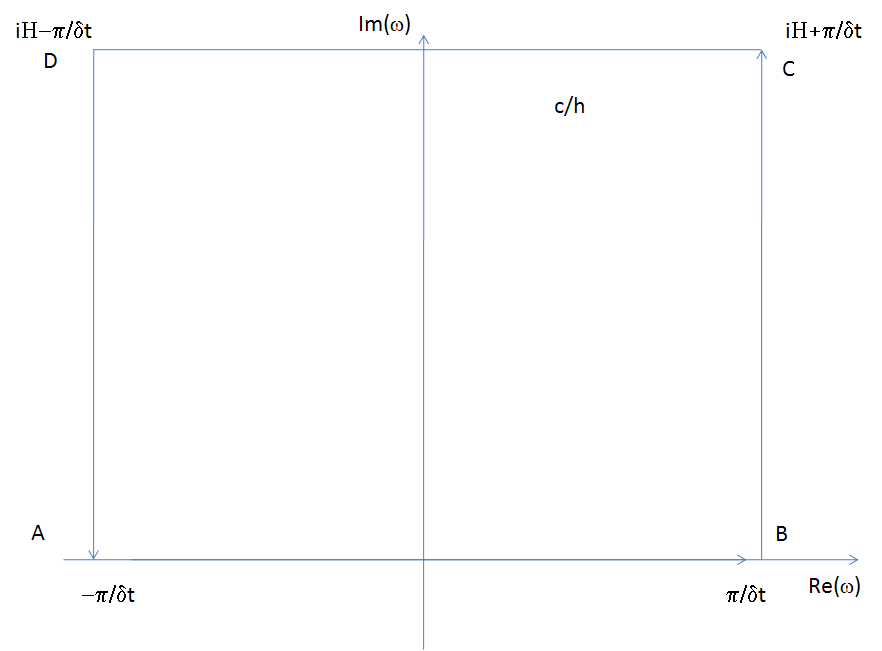}
    \caption{Contour of integration for the integral in (\ref{eq_greenfuncdt}).}
    \label{fig_contour3}
\end{center}
\end{figure}

Let $h = h_m$ and let us introduce the two functions
\begin{align}
\phi_{0}(t, x, \tau) &\equiv 2 {\mathcal L}_m(x, x+\tau h)
e^{ t {\mathcal L}_m(x, x)} 1(t\ge0), \\
\phi_{\delta t}(j, x, \tau) &\equiv 2 {\mathcal L}_m(x, x+\tau h)
\big( 1 + \delta t {\mathcal L}_m(x, x) \big)^{j-1}.
\end{align}
and the corresponding Fourier transforms
\begin{align}
\hat\phi_{0}(\omega, x, \tau) &= \int_0^\infty \phi_{0}(t, x, \tau)
e^{-i\omega t} {d\omega\over 2\pi} =  \left({v(x)\over h^2} + \tau {
\mu(x)\over h}\right) \left( {v(x)\over
h^2} + i \omega \right)^{-1} \\
\hat\phi_{\delta t}(\omega, x, \tau) &= \sum_{j=0}^{t\over \delta t}
\phi_{\delta t}(j, x, \tau) e^{-i\omega j\delta t} =
\left({v(x)\over h^2} + \tau {\mu(x)\over h}\right) \left(
e^{i\omega\delta t} - 1 + \delta t {v(x)\over h^2} \right)^{-1}.
\end{align}
We have that
\begin{align}
\hat\phi_{\delta t}(\omega, x, \tau) &=
\left({v(x)\over h^2} + \tau {\mu(x)\over h}\right)
\left(
i\omega  + {v(x)\over h^2} - {\omega^2 \over 2} \delta t + O(\delta t^2)
\right)^{-1} \notag \\
& = \hat\phi_{0}(\omega, x, \tau) + {\omega^2\over 2 v(x)} h^2 \delta t + O(h^2 \delta t^2).
 = \hat\phi_{0}(\omega, x, \tau) +  O(h^4),
 \end{align}
where the last step uses the fact that $\delta t = O(h^2)$.

Let us also introduce the functions
\begin{align}
\psi_{0}(t, x, \tau) &\equiv e^{ t {\mathcal L}_m(x, x)} 1(t\ge0),
\hskip1cm \psi_{\delta t}(j, x, \tau) \equiv \sum_{k = 1}^j \big( 1
+ \delta t {\mathcal L}_m(x, x) \big)^{j-1}.
\end{align}
and the corresponding Fourier transforms
\begin{align}
\hat\psi_{0}(\omega, x, \tau) &= \left( {v(x)\over h^2} + i \omega
\right)^{-1}, \hskip1cm \hat\psi_{\delta t}(\omega, x, \tau) =
\left( e^{i\omega\delta t} - 1 + \delta t {v(x)\over h^2}
\right)^{-1}.
\end{align}
Again we find that
\begin{align}
\hat\psi_{0}(\omega, x, \tau) = \hat\psi_{\delta t}(\omega, x, \tau)
 +  O(h^4).
\end{align}

If $\gamma$ is a symbolic sequence, then let us set
\begin{align}
\hat W_m(\gamma, q; \omega) &= \hat\psi_{0}(\omega, \gamma_q)
\prod_{j=0}^{q-1} \hat\phi_0(\omega;
\gamma_j, {\rm sgn}(\gamma_{j+1}-\gamma_j))\\
\hat {W}_m^{\delta t}(\gamma, q; \omega)
&=  \hat\psi_{\delta t}(\omega, \gamma_q) \prod_{j=0}^{q-1}
\hat\phi_{\delta t}(\omega; \gamma_j, {\rm sgn}(\gamma_{j+1}-\gamma_j)).
\end{align}
We have that
\begin{align}
G^{\delta t}_{m}(x, y; \omega) - G_m(x, y; \omega) = & {1\over h}
\sum_{q=1}^\infty 2 ^{-q} \sum_{
\begin{matrix}
\gamma\in \Gamma_m : \gamma_0 = x, \gamma_{q} = y\\
\lvert\gamma_j - \gamma_{j-1}\lvert = 1 \forall j\ge1
\end{matrix}}
\left( \hat {W}_m^{\delta t} (\gamma, q; \omega) - \hat W_m (\gamma,
q; \omega)\right). \label{eq_udiff}
\end{align}
The integration over the contour in Fig.  \ref{fig_contour2} can
again be split into an integration over the countour ${\mathcal
C}_-$ and an integration over ${\mathcal C}_+$. The integral over
${\mathcal C}_+$ can be bounded from above thanks to Lemma
\ref{lemma_cplus}. Furthermore,  we have that
\begin{align}
&\bigg\lvert \int_{{\mathcal C}_-} \bigg( G^{\delta t}_{m}(x, y;
\omega) - G_m(x, y; \omega) \bigg)
e^{i\omega t}  {d\omega\over 2\pi} \bigg\lvert \notag \\
&\leq c h^{-1} \sqrt{q_{\rm max}} \max_{
\begin{matrix}
q, \gamma\in \Gamma_m : \gamma_0 = x, \gamma_{q} = y\\
\lvert\gamma_j - \gamma_{j-1}\lvert = 1 \forall j\ge1
\end{matrix}}
\bigg \lvert \int_{{\mathcal C}_-} \bigg( \hat {W}_m^{\delta t}
(\gamma, q; \omega) - \hat W_m (\gamma, q; \omega) \bigg) e^{i\omega
t}  {d\omega\over 2\pi}  \bigg\lvert. \notag
\\
& \leq c h^2 \label{eq_udiff2}
\end{align}

To bound the time derivative, we have to consider
\begin{align}
&\bigg\lvert \int_{{\mathcal C}_-} \bigg( {e^{i\omega\delta t} - 1
\over \delta t} G^{\delta t}_{m}(x, y; \omega) - i \omega G_m(x, y;
\omega) \bigg)
e^{i\omega t}  {d\omega\over 2\pi} \bigg\lvert \notag \\
\end{align}
But, since $\delta t = O(h^2)$, also this difference is $O(h^2)$.

\section{Conclusions}

We obtained bounds on convergence rates for explicit discretization
schemes to the kernel of one-dimensional diffusion equations with
continuous coefficients. We consider both semidiscrete
triangulations with continuous time and explicit Euler schemes with
time step small enough for the method to be stable. The proof is
constructive and based on a new technique of path conditioning for
Markov chains and a renormalization group argument. Convergence
rates depend on the degree of smoothness and H\"older
differentiability of the coefficients. The method is of more general
applicability and will be extended in future work.

\bibliographystyle{giwi}
\bibliography{holder}

\end{document}